\documentclass[letterpaper, 10 pt, conference]{ieeeconf}
\IEEEoverridecommandlockouts     

\usepackage[english]{babel}
\usepackage[noend]{algorithmic}
\usepackage{algorithm}
\usepackage{amsmath,amssymb,amsbsy,latexsym}
\usepackage{mathrsfs}
\usepackage{graphicx}
\usepackage{epstopdf}
\usepackage{multirow}
\usepackage{enumerate}
\usepackage{color}
\usepackage{cite}

\newtheorem{theorem}{Theorem}

\newtheorem{definition}{Definition}

\newtheorem{example}{Example}

\newtheorem{corollary}{Corollary}
\newtheorem{rem}{Remark}
\newtheorem{Procedure}{Procedure}

\newcommand{\QNUM}{n_\mathrm{p}}
\newcommand{\SWS}{\mathfrak{S}}

\newcommand{\rank}{\mathrm{rank}}

\newcommand{\NX}{n_\mathrm{x}}
\newcommand{\NY}{n_\mathrm{y}}
\newcommand{\NU}{n_\mathrm{u}}

\newcommand{\X}{\mathbb{R}^{n_\mathrm{x}}}

\newcommand{\TR}[1]{{\color{black}#1}}

\begin{document}
\title{Minimal realizations of input-output behaviors by LPV state-space representations with affine dependency}
\author{Mih\'aly Petreczky, Roland T\'oth, and Guillaume Merc\`{e}re
\thanks{Mih\'aly Petreczky is with Univ. Lille, CNRS, Centrale Lille, UMR 9189 CRIStAL, F-59000 Lille, France.{\tt\ mihaly.petreczky@centralelille.fr}, 
Guillaume Merc\`{e}re is with the University of Poitiers, Laboratoire d'Informatique et d'Automatique pour les Syst\`emes, 
{\tt guillaume.mercere@univ-poitiers.fr}, 
Roland T\'oth is with the Control Systems Group, Department of Electrical Engineering, Eindhoven University of Technology, Eindhoven, The Netherlands and the Systems and Control Laboratory, Institute for Computer Science and Control, 
Budapest, Hungary {\tt \small r.toth@tue.nl}}
}

\maketitle
\begin{abstract}
The paper makes the first steps towards a behavioral theory of LPV state-space representations with an affine dependency on scheduling, by characterizing minimality of such state-space representations. 
It is shown that minimality is equivalent to observability, and that minimal realizations of the same behavior are isomorphic.
Finally, we establish a formal relationship between minimality of LPV state-space representations with an affine dependence on scheduling
	and minimality of LPV state-space representations with a dynamic and meromorphic dependence on scheduling. 
\end{abstract}

\section{Introduction}

\emph{Linear parameter-varying} (LPV) \cite{Toth2010SpringerBook,Toth11_LPVBehav} 
systems represent a widely used system class which is more general than 
\emph{linear time-invariant} (LTI) systems and which can capture \emph{nonlinear} and \emph{time-varying} behavior.  
LPV systems are modeled by 
linear time-varying difference or differential equations, where the time varying coefficients are 
functions of a time-varying \emph{scheduling} signal.  
LPV systems are widely used in control (\!\!\cite{Rugh00,Packard94CL,Apkarian95TACT,Scherer2009}) and in system identification,  (\!\!\cite{Toth10_SRIV,Giarre2002,Butcher08,Poolla2008,Wingerden09,Sznaier:01,Verdult02}).

Despite these advances, there are still gaps in \textcolor{black}{theory of LPV systems, in particular in their \emph{realization theory}. Realization theory 
aims at characterizing the relationship between the input-output behavior and
certain classes of state-space representations (linear time-invariant, bilinear, etc.), see \cite{Isi:Nonlin,Son:MathContr}.}
Since realization theory is used in system identification and model reduction, \textcolor{black}{and} data-driven control \cite{RolandFundLemma}, filling this gap is important for further development
of those disciplines for LPV systems. 

\textcolor{black}{\textbf{Prior work on realization theory and motivation}}
\begin{color}{black}
Realization theory oi LPV-SS was first addressed
\cite{Toth2010SpringerBook,Toth11_LPVBehav}, where  behavioral theory was used to clarify realization theory, the concepts of minimality and equivalence
for so called LPV state-space representation (\emph{LPV-SS} for short) with the so called \emph{meromorphic} and \emph{dynamical dependency} of the model parameters on the scheduling variable. We will refer to the latter class of LPV-SSs as \emph{meromorphic LPV-SS}.
Recall from \cite{Toth2010SpringerBook,Toth11_LPVBehav} that meromorphic and dynamic dependency means that the entries of the system matrices are  fractions of analytic functions of  the current value and past values of the scheduling variables (in discrete-time), or the high-order 
derivatives of the scheduling variables (in continuous-time).
A major drawback of meromorphic LPV-SSs is that for practical applications it is often preferable to use LPV-SSs with a static and
affine dependency on the scheduling variable (\emph{LPV-SSA} for short), i.e., LPV-SS whose system parameters are affine functions of the instantaneous value of the scheduling variable.
However, whereas LPV-SSAs are a subclass of meromorphic LPV-SSs, the 
system-theoretic transformations (passing from input-output to state-space representation, transforming a state-space representation to a minimal one, etc.)  of \cite{Toth2010SpringerBook,Toth11_LPVBehav} result in meromorphic LPV-SSs, even if applied to LPV-SSAs, 
see \cite{CoxThesis}. 

  In \cite{BelikovSCL2014}, realizability of LPV input-output equations by LPV-SSs 
  with a general (non affine) dependence on the scheduling variable was investigated. However, it is not  clear that all behaviors of interest
  admit the LPV input-output 
  representations from \cite{BelikovSCL2014}, and \cite{BelikovSCL2014} does not address
  minimality.

In \cite{LPVSSAkalman} Kalman-style realization theory for LPV-SSAs was developed. 
A drawback of \cite{LPVSSAkalman} 
lies in the use of input-output functions, which captures the input-output behavior only from a certain fixed initial state. In contrast, for control synthesis,  the initial state is not fixed.  

That is, \cite{Toth2010SpringerBook,Toth11_LPVBehav,BelikovSCL2014} do not address behavioral realization theory for LPV-SSAs.



\textbf{Contribution}
\end{color}
In this paper we make a first step towards a behavioral approach for LPV-SSAs. \textcolor{black}{Similarly to \cite{Toth2010SpringerBook}, we use the
concept of manifest behavior from \cite{PW91}
to formalize the input-output behavior of LPV-SSAs}. 
\textcolor{black}{Then, under suitable assumptions,  
the following counterparts of the well-known results for LTI behaviors \cite{PW91} hold}:
\begin{itemize}
\item
  A LPV-SSA is a minimal realization of a given behavior if and only if it is observable, and all
		minimal realization of the same manifest behavior are related by a  linear (constant) isomorphism. 
 \item
    A behavior is controllable, if and only if its minimal LPV-SSAs is span-reachable
    from the zero initial state. 
\end{itemize}
We also formulate a \textcolor{black}{computionally effective} minimization procedure for LPV-SSAs. 
Furthermore, we show that \textcolor{black}{under some assumptions}, a minimal 
LPV-SSA realizatio of a behavior is also minimal   
if viewed as a meromorphic LPV-SS \cite{Toth2010SpringerBook}.
\textcolor{black}{The latter is interesting, 
as in contrast to meromorphic LPV-SSs, there are 
computationally effective algorithms for minimization and
checking minimality of LPV-SSAs. }

\textbf{Outline}
 In Section \ref{sec:pre}, we present the necessary background on LPV-SSA representations and then formalize several system theoretic concepts.
 In Section \ref{para:main}, the main results are introduced, Section \ref{sect:proof_main} gathers the proofs. 
\section{Preliminaries} \label{sec:pre}
\begin{color}{black}
Let $\mathbb{T}=\mathbb{R}_{0}^{+}=[0,+\infty)$ be the time axis in the  \emph{continuous-time} (CT) case and $\mathbb{T}=\mathbb{N}$ in the  \emph{discrete-time} (DT) case.  Let $\xi$ be the differentiation operator $\frac{d}{dt}$ (in CT) and the forward time-shift
operator $q$ (in DT), \emph{i.e.},
if $z: \mathbb{T} \rightarrow \mathbb{R}^{n}$,  then
$(\xi z)(t)=\frac{d}{dt} z(t)$, if  CT, and 
$(\xi z)(t)=z(t+1)$, if  DT. 
\end{color}

Define a LPV state-space representations with \emph{affine} dependence on the \emph{scheduling variable} (\emph{LPV-SSA}) as
\begin{equation}\label{equ:alpvss}
  \Sigma \ \left\{
  \begin{array}{lcl}
   \xi x (t) &=& A(p(t)) x(t) + B(p(t)) u(t), \\
   y(t)  &=& C(p(t)) x(t) + D(p(t))u(t),
  \end{array}\right.
\end{equation}
where \textcolor{black}{$x: \mathbb{T} \rightarrow  \mathbb{R}^{n_\mathrm{x}}$ is the state trajectory, $y:\mathbb{T} \rightarrow \mathbb{R}^{n_\mathrm{y}}$ is the (measured) output trajectory, $u: \mathbb{T} \rightarrow \mathbb{R}^{n_\mathrm{u}}$ is the (control) input signal and $p:\mathbb{T} \rightarrow  \mathbb{P}\subseteq \mathbb{R}^{n_\mathrm{p}}$ is the so called \emph{scheduling signal} of the system represented by $\Sigma$.  }
\begin{color}{black}
Moreover, $A,B,C,D$ are matrix valued affine functions defined on $\mathbb{P}$,
i.e. there exists matrices $A_i \in \mathbb{R}^{\NX \times \NX}$,
$B_i \in \mathbb{R}^{\NX \times \NU}$, $C_i \in \mathbb{R}^{\NY \times \NX}$ and
$D_i \in \mathbb{R}^{\NY \times \NU}$ for all $i=0,1,\ldots,n_{\mathrm p}$, such that
\begin{equation*}
\begin{split} 
  A(\mathbf{p}) = A_0 + \sum_{i=1}^{\QNUM} A_i \mathbf{p}_i, ~
 B(\mathbf{p}) = B_0 + \sum_{i=1}^{\QNUM} B_i \mathbf{p}_i \\
 C(\mathbf{p}) = C_0 + \sum_{i=1}^{\QNUM} C_i \mathbf{p}_i, ~ 
 D(\mathbf{p}) = D_0 + \sum_{i=1}^{\QNUM} D_i \mathbf{p}_i
\end{split}
\end{equation*}
for every $\mathbf{p}=[\begin{array}{ccc} \mathbf{p}_1& \ldots & \mathbf{p}_{\QNUM}\end{array}]^\top \in \mathbb{P}$
\footnote{Note that in the sequel we use italic letters to denote scheduling signals, and boldface letters to denote elements of $\mathbb{P}$.}.
\end{color}

\begin{color}{black}
Note that in LPV systems, 
the input and scheduling signals play the
role of exogenous inputs.
Note that it is often assumed that the scheduling signals are bounded to ensure desirable properties, e.g., stability,
hence in general
$\mathbb{P} \ne \mathbb{R}^{n_{\mathrm p}}$.

\end{color}
In the sequel, we use the shorthand notation 
\begin{equation*}
   \Sigma=(\mathbb{P},\left\{ A_i, B_i, C_i, D_i \right\}_{i=0}^{\QNUM})
\end{equation*}
to denote a LPV-SSA of the form \eqref{equ:alpvss} and 
use  $\dim{( \Sigma )}=n_\mathrm{x}$ to denote its state dimension. 

\begin{color}{black}
For the purposes of this paper we need to formalize the solution concept for LPV-SSAs.  
To this end, we define the sets $\mathcal{X},\mathcal{Y},\mathcal{U},\mathcal{P}$ of  respectively state and output trajectories,  and input and  scheduling signals as follows. 
For a set $X$ let $X^{\mathbb{N}}$ denote the set of all functions of the form $f:\mathbb{N} \rightarrow X$. 
In DT, let $\mathcal{X}=(\mathbb{R}^{\NX})^\mathbb{N}, \mathcal{Y}=(\mathbb{R}^{\NY})^\mathbb{N}, \mathcal{U}=(\mathbb{R}^{\NU})^{\mathbb{N}},\mathcal{P}=\mathbb{P}^{\mathbb{N}}$. In CT, let us denote by $\mathcal{C}_\mathrm{p}(\mathbb{R}_{0}^{+},\mathbb{R}^n)$ the set of all functions of the form $f:\mathbb{R}_{0}^{+} \rightarrow \mathbb{R}^{n}$ which are
piecewise-continuous. In addition let
$\mathcal{C}_\mathrm{a}(\mathbb{R}_{0}^{+},\mathbb{R}^n)$ 
be the set of all absolutely continuous functions of the form $f:\mathbb{R}_{0}^{+} \rightarrow \mathbb{R}^{n}$.
Then, in CT let 
$\mathcal{X}=\mathcal{C}_\mathrm{a}(\mathbb{R}_{0}^{+},\mathbb{R}^{\NX}), \mathcal{Y}=\mathcal{C}_\mathrm{p}(\mathbb{R}_{0}^{+},\mathbb{R}^{\NY}), \mathcal{U}=\mathcal{C}_\mathrm{p}(\mathbb{R}_{0}^{+},\mathbb{R}^{\NU}), \mathcal{P}=\mathcal{C}_\mathrm{p}(\mathbb{R}_{0}^{+},\mathbb{P})$.
\end{color}

 By a \textcolor{black}{solution}  of $\Sigma$ we mean a tuple of trajectories $(x,y,u,p)\in(\mathcal{X},\mathcal{Y},\mathcal{U},\mathcal{P})$ satisfying \eqref{equ:alpvss} for almost all $t \in \mathbb{T}$ in CT case, and for all $t \in \mathbb{T}$ in DT. 

Notice that without loss of generality, the solution trajectories in CT can be considered on the half line $\mathbb{R}_0^+$ with
$t_\mathrm{o}=0$. 
Note that for any input and scheduling signal  $(u,p)\in\mathcal{U}\times \mathcal{P}$ and any initial state \textcolor{black}{$x_{\mathrm o} \in \mathbb{R}^{\NX}$}, there exists a \emph{unique} pair $(y,x)\in\mathcal{Y}\times \mathcal{X}$ such that $(x,y,u,p)$ is a solution of  \eqref{equ:alpvss} and $x(0)=x_{\mathrm o}$,
see \cite{Toth2010SpringerBook}.
Next,  inspired by \cite{PW91,Toth2010SpringerBook}, we define the notion of manifest behavior for LPV-SSAs.
\begin{definition}
 A \textcolor{black}{\emph{manifest behavior}} is a subset $\mathcal{B} \subseteq \mathcal{Y} \times \mathcal{U} \times \mathcal{P}$.
	The \textcolor{black}{\emph{manifest behavior} $\mathcal{B}(\Sigma)$ of a LPV-SSA $\Sigma$} is defined as
 \begin{equation*} 
 \begin{split}
 \textcolor{black}{\mathcal{B}}(\Sigma)=\bigl\{(y,u,p) \in \mathcal{Y} \times \mathcal{U} \times \mathcal{P} \mid \exists x \in \mathcal{X}  \\
 \mbox{ s.t. } (x,y,u,p) \mbox{ is a solution of \eqref{equ:alpvss}}
    \bigr\}
\end{split}
 \end{equation*}
  
	The LPV-SSA  $\Sigma$ is a \emph{realization} of a manifest behavior  \textcolor{black}{$\mathcal{B}\subseteq \mathcal{Y}\times \mathcal{U} \times \mathcal{P}$}, 
 if \textcolor{black}{$\mathcal{B}=\mathcal{B}(\Sigma)$}.
\end{definition}
That is, the manifest behavior of a LPV-SSA $\Sigma$ is the set of all tuples $(y,u,p)$ such that $\Sigma$ generates the output $y$ for some initial state, if $\Sigma$ is fed input $u$ and scheduling $p$. 
The corresponding definition of minimality is then as follows. 
\begin{definition}
	A LPV-SSA $\Sigma$ is a \emph{minimal realization} of a manifest behavior 
	$\mathcal{B}$, if it is a realization of $\mathcal{B}$, and for any  LPV-SSA $\Sigma^{'}$ such that $\Sigma^{'}$ is a realization of $\mathcal{B}$, $\dim \Sigma \le \dim \Sigma^{'}$. 
	We say that $\Sigma$ is \emph{minimal}, if $\Sigma$ is a minimal realization of its own manifest behavior $\mathcal{B}(\Sigma)$. 
\end{definition}
Manifest behaviors are a natural formalization of the intuition behind input-output behaviors of LPV-SSAs. 
However, input-output behaviors can also be formalized  using \textcolor{black}{input-output functions}. 
The latter was used 
in \cite{LPVSSAkalman}  for proposing a Kalman-style realization theory for
LPV-SSAs. The principal definitions are as follows:
\begin{definition}
	Let  $x_\mathrm{o} \in \mathbb{X}$ be an initial state of $\Sigma$. Define the \textcolor{black}{\emph{input-output (i/o) function}
$\mathfrak{Y}_{\Sigma,x_\mathrm{o}}:   \mathcal{U} \times \mathcal{P}  \rightarrow \mathcal{Y}$},
induced by the initial state $x_{\mathrm o}$, as follows
\textcolor{black}{: for any $(u,p) \in \mathcal{U} \times \mathcal{P}$,
$y=\mathfrak{Y}_{\Sigma,x_\mathrm{o}}(u,p)$ holds if and only if there exists a solution $(x,y,u,p)$ of  \eqref{equ:alpvss} 
such that $x(0)=x_{\mathrm o}$}.
\end{definition}
\begin{definition}
	A LPVS-SSA $\Sigma$ is a \emph{realization} of
	an i/o function $\mathfrak{F} : \mathcal{U} \times \mathcal{P} \rightarrow \mathcal{Y}$ 
	from the initial state $x_{\mathrm o} \in \mathbb{X}$, if \textcolor{black}{$\mathfrak{F}$ coincides with the i/o function of $\Sigma$ induced by $x_{\mathrm o}$, i.e. }
	$\mathfrak{F}=\mathfrak{Y}_{\Sigma,x_{\mathrm o}}$.
     \textcolor{black}{We say $\Sigma$ is a \emph{realization} of $\mathfrak{F}$, if it is a 
	realization of $\mathfrak{F}$ from some initial state.}
 
	\textcolor{black}{The LPV-SSA $\Sigma$ is a \emph{minimal realization of $\mathfrak{F}$} 
   if it is a realization of $\mathfrak{F}$,
    and for every LPV-SSA $\Sigma^{\prime}$  which is a realization of $\mathfrak{F}$, $\dim{ ( \Sigma ) } \leq \dim{ ( \Sigma^\prime ) }$.}
\end{definition}
\begin{color}{black}
An drawback of using i/o functions instead of manifest behaviors is that the former capture the input-output behavior for one choice of initial states.  However, we can account for all initial states   by using families of i/o functions: 
\begin{definition}
	If $\Sigma$ is a LPV-SSA with the state-space $\mathbb{R}^{\NX}$, then the set $\mathbb{F}(\Sigma)=\{\mathcal{Y}_{\Sigma,x_{\mathrm o}} \mid x_{\mathrm o} \in \mathbb{R}^{\NX}\}$ of all i/o functions of $\Sigma$
    induced by some initial state of $\Sigma$ 
    is called the \emph{family of i/o functions}  of $\Sigma$.
	A family $\Phi$ of i/o functions of the form $\mathfrak{F}:\mathcal{U} \times \mathcal{P} \rightarrow \mathcal{Y}$ is
	\emph{realized} by $\Sigma$, if $\Phi=\mathbb{F}(\Sigma)$.
\end{definition}
It is natural to ask if using families of i/o functions are  equivalent to using LPV manifest behaviors.
Clearly, if $\mathbb{F}(\Sigma)=\mathbb{F}(\hat{\Sigma})$, then $\mathcal{B}(\Sigma)=\mathcal{B}(\hat{\Sigma})$ holds. 
\end{color}
In fact, the example below shows that the converse is not true.
\begin{example}
\label{rem:iobehav}
	Consider the LPV-SSAs $\Sigma$ and $\Sigma^{'}$ 
	\[ 
         \begin{split} 
		 &	 \Sigma \left \{ \begin{array}{l}   x_1(t+1)=x_1(t)+p(t)x_2(t), ~ x_2(t+1)=0, \\
 	         y(t)=x_1(t)+p(t)x_2(t),  \end{array}\right. \\
	  & \Sigma^{'} \left \{ \begin{array}{l}  z(t+1)=0, ~ y(t)=z(t),  \end{array}\right. 	 
	 \end{split}	  
	\]	      
	with the scheduling space $\mathbb{P}=\mathbb{R}$. 
 A straightforward calculation reveals that
 $\mathcal{B}(\Sigma^{'})=\mathcal{B}(\Sigma)$.
Indeed,	note that $(x,y,u,p)$ is a solution of $\Sigma$, if and only if $x_1(t)=x_1(0)+p(0)x_2(0)$ and $x_2(t)=0$ for $t > 0$, hence
	$y(t)= x_1(0)+p(0)x_2(0)$ for all $t \ge 0$. Therefore, $\mathcal{B}(\Sigma)=\{ (y,u,p) \mid y(t) \mbox{ is constant} \}$. 
	It then follows that $\mathcal{B}(\Sigma^{'})=\mathcal{B}(\Sigma)$. 

 However,  for $x_{\mathrm o}=\begin{bmatrix} 1 & 1 \end{bmatrix}^{\top}$ and $x_{\mathrm o}^{'}=1$,
 $\mathfrak{Y}_{\Sigma,x_{\mathrm o}} \ne \mathfrak{Y}_ {\Sigma^{'},x_{\mathrm o}^{'}}$. To see this,
 it is enough to evaluate the i/o functions 
 involved for any two scheduling signals $p_1$ and $p_2$
 such that $p_1(0)=0$ and $p_1(1)=1$.
 Hence, $\mathbb{F}(\Sigma) \ne \mathbb{F}(\Sigma^{'})$.
 Indeed,  $\mathfrak{Y}_{\Sigma,x_{\mathrm o}}(0,p)=1=\mathfrak{Y}_ {\Sigma^{'},x_{\mathrm o}^{'}}(0,p)=x_{\mathrm o}^{'}$. However, if $p(0)=1$, then 
	$\mathfrak{Y}_{\Sigma,x_{\mathrm o}}(0,p)=2 \ne 1=\mathfrak{Y}_{\Sigma^{'},x_{\mathrm o}^{'}}(0,p)$.

%
\end{example}
\begin{color}{black}
That is, realization theory of manifest behaviors is not equivalent to that of i/o functions.
In particular, the minimality results of \cite{LPVSSAkalman}
do not apply in the behavioral setting. 
It is then a natural to ask if similarly to \cite{LPVSSAkalman}, observability and span-reachability to characterize minimality in the behaviroal setting.
\end{color}
	The latter are recalled below.
\begin{definition}
\label{def:reachobs}
	Let $\Sigma$ be a LPV-SSA of the form~\eqref{equ:alpvss}.  $\Sigma$ is \emph{span-reachable} from an initial state $x_\mathrm{o} \in \X$, \textcolor{black}{if  the linear  span of all states reachable from
 $x_{\mathrm o}$ equals the whole state-space $\mathbb{R}^{\NX}$, i.e.
	$\mathrm{Span}\{ x(t) \mid (x,y,u,p) \in \mathcal{X} \times \mathcal{Y} \times \mathcal{U} \times \mathcal{P}, (x,y,u,p) \mbox{ is a solution of \eqref{equ:alpvss}}, t \in \mathbb{T}, x(0)=x_{\mathrm o} \}=\mathbb{R}^{\NX}$}.
 
 The LPV-SSA $\Sigma$ is \emph{observable} if any two \textcolor{black}{distinct initial states}  \textcolor{black}{induce
 distinct i/o functions, i.e.  
 $\forall ~ x_1,x_2 \in \mathbb{R}^{\NX}$: $x_1 \ne x_2 \implies \mathfrak{Y}_{\Sigma,x_1} \ne  \mathfrak{Y}_{\Sigma,x_2}$}.
\end{definition}
 Observability and span-reachability can be characterized by rank conditions 
 \cite{LPVSSAkalman}.
  Finally, similarly to \cite{LPVSSAkalman }, we  \textcolor{black}{we would like minimal realizations of the same manifest behavior to be isomorphic. The latter  notion is  defined below} 
  \begin{definition}
 \label{isomorphism}
   \textcolor{black}{Let $\Sigma$ be of the form \eqref{equ:alpvss} and let 
   and let
	  $\Sigma^\prime=(\mathbb{P}, \{ A_i^{'}, B_i^{'}, C_i^{'}, D_i^{'} \}_{i=0}^{\QNUM})$ be a LPV-SSA} with $\dim(\Sigma)=\dim(\Sigma^\prime)=n_\mathrm{x}$. 
	  A nonsingular matrix $T \in \mathbb{R}^{n_\mathrm{x} \times n_\mathrm{x}}$ \textcolor{black}{is \emph{isomorphism}} from $\Sigma$ to $\Sigma^\prime$, if
 \[ A^\prime_{i} T= T A_i ~~ B^\prime_{i}= T B_i ~~  C^\prime_{i}T= C_i ~~   D^\prime_i=D_i, ~ i=0,\ldots, n_{\mathrm p} \]
  \end{definition}
  \textcolor{black}{Note that the matrix $T$ in the definition above does not depend on the scheduling signal, and it acts only on the states of the LPV-SSAs involved. In particular, the LPV-SSAs $\Sigma$ and $\Sigma^{\prime}$ have the same inputs and outputs and are defined over the same set of scheduling signals.}
\begin{color}{black}

\textbf{Problem formulation:} in this paper we will address the following questions: 
 
 \textbf{(1)}
  If two LPV-SSAs have the same manifest behavior,  do they have the same set of i/o functions ?
    
\textbf{(2)}
  Can we characterize minimal LPV-SSAsin terms of observability and span-reachability ?  
  
\textbf{(3)}
  Are minimal LPV-SSA realizations of the same manifest LPV behavior isomorphic ? 
  
\textbf{(4)}
    Is there an algorithm for transofrming an LPV-SSA to a 
     minimal LPV-SSA realization of its manifest behavior. 

 \textbf{(5)}
   Are minimal LPV-SSAs also minimal
   as meromorphic LPV-SS from \cite{Toth2010SpringerBook,Toth11_LPVBehav} ?

\end{color}

\section{Main results}
\label{para:main}
 In this section, we present the main results of the paper, which answer the questions formulated above.

\subsection{Input-output functions vs. behaviors}
\label{sec:iofunbehav}
 We start by clarifying the relationship between  manifest behaviors and  i/o functions of LPV-SSAs. 
 \textcolor{black}{To this end, we need the following definition.}
 \begin{definition}
	A LPV-SSA of the form \eqref{equ:alpvss} is said to satisfy the \emph{regularity certificate (RC)} if
    \begin{color}{black}
		 \textbf{(1)} $\mathbb{P}$ is convex with non-empty interior, and, in addition,
		 \textbf{(2)}
   in the DT case, the matrix $A(\mathbf{p})$ is invertible for all $\mathbf{p} \in \mathbb{P}$.
   \end{color}
\end{definition}
\textcolor{black}{In CT, the satisfaction of RC depends only on $\mathbb{P}$, and it is satisfied if $\mathbb{P}$ is a Cartesian product of intervals, e.g.,
$\mathbb{P}=[a,b]^{n_{\mathrm p}}$, $a < b$.
In DT, the RC condition is more restrictive.}
 

\begin{theorem}
\label{col:behavior}
    \textcolor{black}{Let $\Sigma$ and $\hat{\Sigma}$ be two LPV-SSAs which satisfy RC.
 Then $\Sigma$ and $\hat{\Sigma}$ have the same family of i/o functions,i.e. $\mathbb{F}(\Sigma)=\mathbb{F}(\hat{\Sigma})$, if and only if their manifest behavior is the same, i.e., $\mathcal{B}(\Sigma_1)=\mathcal{B}(\Sigma_2)$.}
\end{theorem}
The proof of Theorem \ref{col:behavior} is presented in Section \ref{sect:proof_main}.
The theorem above says that
\textcolor{black}{manifest behaviors and families of i/o function are equivalent formalizations of input-output behaviors of LPV-SSA satisfying RC. }
\textcolor{black}{Theorem \ref{col:behavior} is no longer
true if we drop RC, see Example \ref{rem:iobehav}.}

\subsection{Minimality}
\label{sect:minim}
 Theorem \ref{col:behavior} and an extension of the  results of \cite{LPVSSAkalman} to families of i/o functions
 leads to the following characterization of minimal realizations of manifest behaviors.
\begin{theorem}
\label{the:behavior_min1}
	\textcolor{black}{A LPV-SSA which satisfies RC 
	is minimal, if and only if it is observable.}
	Furthermore, \textcolor{black}{any two minimal LPV-SSAs which satisfy RC and which are realizations of the same manifest behavior are isomorphic.}
\end{theorem}
The proof of Theorem \ref{the:behavior_min1} is presented in Section \ref{sect:proof_main}.
Note that minimality of LPV-SSAs does not require span-reachability. This is in contrast with  minimal LPV-SSA realizations of i/o functions,
but this is consistent with the classical results for  LTI systems \textcolor{black}{\cite{PW91}}.

Theorem \ref{the:behavior_min1} suggests a minimization procedure based on the observability reduction procedure 
from \cite{MertLPVCDC2015,LPVSSAkalman}. We  recall the latter below.
 Let $\Sigma$ be an LPV-SSA of the form \eqref{equ:alpvss}, and recall from \cite{LPVSSAkalman}
 the definition of extended $n$-step observability matrices $\mathcal{O}_n$ of $\Sigma$, $n \in \mathbb{N}$:
  \begin{align*}
	  \mathcal{O}_0 &=     \left[\begin{array}{cccc}  C_0^\top &  \cdots & C_{n_\mathrm{p}}^\top \end{array}\right]^\top \\
    \mathcal{O}_{n+1} &=
    \left[\begin{array}{cccc}  
      \mathcal{O}_n^\top & 
      A_0^\top \mathcal{O}_{n}^\top  &
     \cdots & A_{n_\mathrm{p}}^\top \mathcal{O}^\top_n \end{array}\right]^\top.
 \end{align*}
\textcolor{black}{By \cite{LPVSSAkalman},} $\Sigma$ is observable, if and only if $\mathrm{rank} ~(\mathcal{O}_{n_{\mathrm x}-1})=n_{\mathrm x}$.
\begin{Procedure}[Observability reduction]
\label{LSSobs}
	Consider the matrix $T=\begin{bmatrix} b_1 & b_2 & \ldots & b_{n_{\mathrm{x}}} \end{bmatrix}^{-1}$,  where
	$\{b_i\}_{i=1}^{n_\mathrm{x}} \subset \mathbb{R}^{n_\mathrm{x}}$ is a basis such that  $\mathrm{Span}\{b_{o+1},\ldots,b_{n_\mathrm{x}} \}= \mathrm{Ker} \{ \mathcal{O}_{n_\mathrm{x}-1}\} $, 
		In the new basis, the matrices
$\{A_i,B_i,C_i\}_{i=0}^{n_\mathrm{p}}$ become
\begin{align*}
	TA_iT^{-1}=\begin{bmatrix} A_{i}^{\mathrm O} & 0 \\ A^\prime_{i} & A^{\prime\prime}_{i} \end{bmatrix}, ~
		TB =\begin{bmatrix}  B_{i}^{\mathrm O} \\ B_i^{\prime} \end{bmatrix}, ~
	C_i T^{-1}=\begin{bmatrix} C_i^{\mathrm O} & 0 \end{bmatrix}, 
\end{align*} 
where $A^{\mathrm O}_{i} \in \mathbb{R}^{o \times o}, B_i^{\mathrm O}
\in \mathbb{R}^{o \times n_\mathrm{u}}$ and $C_i^{\mathrm O} \in \mathbb{R}^{n_\mathrm{y}
  \times o}$.  Define $\Sigma^{\mathrm O}= (\mathbb{P},\{A_i^{\mathrm O}, B_i^{\mathrm O}, C_i^{\mathrm O},D_i\}_{i=0}^{n_\mathrm{p}})$. 
\end{Procedure}
\textcolor{black}{Procedure \ref{LSSobs} is similar to the well-known observability reduction for LTI/bilinear systems, and it can readily be implemented numerically, e.g., see \cite[Remark 2]{MertLPVCDC2015}}
\begin{rem}
\label{LSSobs:rem}
	\textcolor{black}{By \cite{LPVSSAkalman},
    $\Sigma^{\mathrm O}$ is observable. }
    \textcolor{black}{Let $\Pi \in \mathbb{R}^{o \times n_{\mathrm x}}$ 
    be such that 
    $\Pi z$ is formed by the first $o$ elements of $Tz$.}
       Then  $(x,y,u,p)$ is a solution of $\Sigma$, if and only if $(\Pi x,y,u,p)$ is a solution of $\Sigma_{\mathrm O}$.
	Hence, $\mathcal{B}(\Sigma)=\mathcal{B}(\Sigma_{\mathrm O})$. 
	\textcolor{black}{Moreover, for any initial state $x_{\mathrm o}$ of $\Sigma$}, 
	$\mathfrak{Y}_{\Sigma,x_{\mathrm o}}=\mathfrak{Y}_{\Sigma_{\mathrm O},\Pi (x_{\mathrm o})}$.
%
%
%
%
\textcolor{black}{Furthermore, if $\Sigma$ satisfies RC, then so does $\Sigma_{\mathrm O}$}. For CT, there is nothing to show. 
	For DT, notice that $A_{\mathrm O}(\mathbf{p})$, $\mathbf{p} \in \mathbb{P}$ is the upper left block of 
    the triangular matrix
	$\hat{A}(\mathbf{p})=TA(\mathbf{p})T^{-1}$, hence if $A(\mathbf{p})$ is invertible, then so is $A_{\mathrm O}(\mathbf{p})$.
\end{rem}
\begin{corollary}[Minimization]
	\textcolor{black}{If $\Sigma$ satisfies RC, then $\Sigma_{\mathrm O}$ returned
	by Procedure \ref{LSSobs} satisfies RC, it is minimal and it has the same manifest behavior as $\Sigma$. }
\end{corollary}
\textcolor{black}{As in the LTI case, span-reachability is necessary for minimality} of LPV-SSA realizations of \emph{controllable} behaviors, defined similarly to \textcolor{black}{\cite{PW91}}.
\begin{definition}
	The manifest behavior $\mathcal{B}$ is \emph{controllable}, if for any two  \textcolor{black}{elements} $(y_1,u_1,p_1)$, $(y_2,u_2,p_2)$ of $ \mathcal{B}$ and any time instance $t  \in \mathbb{T}$, there exists an 
    \textcolor{black}{element}
	$(y,u,p)$ of $\mathcal{B}$ and a time instance $\tau > 0$, such that 
	$(y|_{[0,t]},u|_{[0,t]},p|_{[0,t]})=(y_1|_{[0,t]},u_1|_{[0,t]},p_1|_{[0,t]})$, and  for all $s \in \mathbb{T}$, $s \ge t+\tau$,
	$(y(s),u(s),p(s))=(y_2(s-t-\tau),u_2(s-t-\tau),p_2(s-t-\tau))$.
\end{definition}

Intuitively, a behavior is controllable, if any i/o trajectory generated by the system up to some time 
can be continued by any other admissible i/o trajectory. 
\begin{theorem}
\label{the:behavior_min2}
        Let $\mathcal{B}$ is a manifest behavior, and let $\Sigma$ be a LPV-SSA which satisfies RC.
	If $\mathcal{B}$ is a controllable, then  $\Sigma$ is a minimal realization of $\mathcal{B}$ if and only if
	$\Sigma$ is span-reachable from zero and observable. 
	Conversely, if $\Sigma$ is span-reachable from zero, then its manifest behavior
	$\mathcal{B}(\Sigma)$ is controllable. 
\end{theorem}
The proof of Theorem \ref{the:behavior_min2} is presented in Section \ref{sect:proof_main}.
Recall from \cite{LPVSSAkalman} that a LPV-SSA is a minimal realization of an i/o function from the zero initial state, if and only if it is observable and span-reachable from zero.  
\textcolor{black}{Theorem \ref{the:behavior_min2} says that
for LPV-SSAs  which satisfy RC are minimal realizations of an i/o function from the zero initial state, if and only if
they minimal realizations of their own manifest behaviors.  }


\subsection{Relationship with the prior results}
Below we show that 
Theorem \ref{the:behavior_min1}-\ref{the:behavior_min2} are consistent with the results of \cite{Toth2010SpringerBook}. 
To this end, recall that LPV-SSAs are special cases of meromorphic LPV-SSs. Recall from \cite{Toth2010SpringerBook} the notions of 
structural state-observability and structural state-reachability and state-trimness \textcolor{black}{and minimality}.
\begin{theorem}
\label{min:compare}
 If $\Sigma$ is a LPV-SSA which satisfies RC, then it is state-trim and the following holds:
 \begin{itemize}
\item if it is observable, then it is structurally state-observable. 
\item if is  span-reachable from $x_\mathrm{o}=0$, then it is  structurally state-reachable.
\item if it is a minimal, then it is a minimal dimensional meromorphic LPV-SS in the sense of \cite{Toth2010SpringerBook}.
 \end{itemize}
\end{theorem}
The proof is presented in Section \ref{sect:proof_main}.
Note that Theorem \ref{min:compare} ceases to be true, if $\Sigma$ does not satisfy RC, see 
\cite[Example 4.1]{CoxThesis} for a counter-example.
\textcolor{black}{In general, there is a tradeoff between the dimensionality of LPV-SSs and the dependence on the scheduling variable (meromorphic, affine), \cite{CoxThesis}. However, for LPV-SSAs which satisfy RC, there is no such tradeoff, i.e. the algorithms of \cite{Toth2010SpringerBook,Toth11_LPVBehav} will not result in smaller state-space representations when applied to such LPV-SSAs. However, they may still introduce meromorphic dependencies on the schedulung variable, see the example below.}
\begin{example}
	Consider in DT the LPV-SSA $\Sigma$ of the form \eqref{equ:alpvss}, such that $\mathbb{P}=[0,1]$ and 
	\[
	   \begin{split}
		   & A_0=\begin{bmatrix}
		    1  &  -2 &   -2 \\
		    0  &    2  &   1 \\
		    -2 &    1 &    2 
		 \end{bmatrix}, ~
	    A_1=\begin{bmatrix}
		   1   & -1 &   -1 \\
		   -1  &   2 &    0 \\
		    -1 &    0 &    2
	    \end{bmatrix} \\
	   &  B_0=\begin{bmatrix} 1 \\ -1 \\ -1 \end{bmatrix},  ~  B_1=\begin{bmatrix} 2 \\ -2 \\ -2 \end{bmatrix} \\
		   &  C_0=\begin{bmatrix} 1 \\ 0 \\ 0 \end{bmatrix}^{\top}, ~ C_1=\begin{bmatrix} 0 \\ 1 \\ 1 \end{bmatrix}^{\top}	   
           \end{split}
	\]
      It follows that 
      \[ 
        \begin{split}
		& A(\mathbf{p})=\begin{bmatrix}  \mathbf{p} + 1 &  -\mathbf{p} - 2 &  - \mathbf{p} - 2 \\
	      -\mathbf{p}, 2\mathbf{p} + 2 &        1 \\
		- \mathbf{p} - 2 &        1 & 2\mathbf{p} + 2 \\
      \end{bmatrix}, \\
		& B(\mathbf{p})=\begin{bmatrix} 2\mathbf{p}+1 \\ -2\mathbf{p}-1 \\ -2\mathbf{p}-1 \end{bmatrix},  ~
			C(\mathbf{p})=\begin{bmatrix} 1 \\ \mathbf{p} \\ \mathbf{p} \end{bmatrix}^{\top}	      
	\end{split}		
  \]
  It follows that the determinant of $A(\mathbf{p})$ is $-2\mathbf{p}^2-3\mathbf{p}-1$ and the latter polynomial is non-zero on $[0,1]$.
  That is, in DT, $\Sigma$ satisfies RC. If we apply the observability reduction procedure Procedure \ref{LSSobs} to $\Sigma$, then
  we obtain the following minimal LPV-SSA $\Sigma_m=(\mathbb{P}, \{A_i^m,B_i^m,C_i^m\}_{i=0}^{1})$,
  \[
     \begin{split}
	     &      A_0^m=\begin{bmatrix} 1 & -2 \\ -2 & 3 \end{bmatrix}, ~	
      A_1^m=\begin{bmatrix} 1 & -1 \\ -2 & 2 \end{bmatrix}, ~	 \\
	      &      B_0^m=\begin{bmatrix} 1 \\ 2 \end{bmatrix}, ~ 	      
		      B_1^m=\begin{bmatrix} 2 \\ -4 \end{bmatrix},  ~ 
			      C_0^m=\begin{bmatrix} 1 \\ 0 \end{bmatrix}^{\top}, ~ 	      
			      C_1^m=\begin{bmatrix} 0 \\ 1 \end{bmatrix}^{\top}
     \end{split}
  \]	  
  which has the same manifest behavior as $\Sigma$.

  Let us view $\Sigma$ as meromorphic LPV-SS and let us minimize $\Sigma$ using \cite{Toth2010SpringerBook,Toth11_LPVBehav}.
  Recall the notation \cite{Toth2010SpringerBook,Toth11_LPVBehav}, in particular, the $\diamond$ symbol for DT.
  In particular, for a will consider matrix valued functions $S$ defined on $\mathbb{P}^k$ and whos entries are meromoprhic functions, 
  for any scheduling signal
  $p \in \mathcal{P}$, $(S \diamond p)$ is a function defined on the time axis $\mathbb{T}$, 
  such that for any $t \in \mathbb{T}$, $(S \diamond p)(t)=S(p(t),\xi p(t),\ldots, \xi^{k-1} p(t))=S(p(t),p(t+1),\ldots,p(t+k-1))$.
  Recall that $\xi$ is the forward time shift. 
  In particular, $(S \diamond p)$ can be viewed as an expression in $p,\xi p, \xi^2 p, \ldots$, see
  \cite{Toth2010SpringerBook}.

  Wit this notation,  for any $p \in \mathcal{P}$ the $3$-step observability matrix is 
 \[ (O \diamond p)=
\begin{bmatrix}
	1 &   p  & p \\
	(O \diamond p)_{2,1} & 	(O \diamond p)_{2,2}  & (O \diamond p)_{2,3} \\
	(O \diamond p)_{3,1} & 	(O \diamond p)_{3,2}  & (O \diamond p)_{3,3} 

 \end{bmatrix}
\]
where 
\[ 
 \begin{split}
  &    (O \diamond o)_{2,1}=-(2\xi p - 1)(p + 1) \\
  &   (O \diamond p)_{2,2}=3\xi p - p + 2 p \xi p - 2  \\
  & (O \diamond p)_{2,3}=3 \xi p - p + 2 p \xi p - 2 \\
 & (O \diamond p)_{3,1}=(p + 1)(3 \xi p - 8\xi^2p - 6p \xi^2p + 5) \\
 & (O \diamond p)_{3,2}=(O \diamond p)_{3,3}=13 \xi^2 p - 5\xi p - 5p - 3p \xi p + \\ 
	 & + 8p \xi^2p + 10 \xi p \xi^2 p + 6p \xi p \xi^2p - 8
 \end{split}
 \]
Recall that $p,\xi p,\xi^2p$ etc. are all viewed as formal variables when defining the observability matrix. 
Over the field of meromoprhic functions, then  kernel of $(O \diamond p)$ is spanned by $\begin{bmatrix} 0 & 1 & -1 \end{bmatrix}^{\top}$. 
Consider the state-space transformation
\[ 
   (T \diamond p)=\begin{bmatrix} 1 &   0 & 0 \\  
	   0 &  1 & 1 \\
	   (t_{31} \diamond p) &  (t_{32} \diamond p) & -1 
   \end{bmatrix}^{-1}
\]   
where  $(t_{31} \diamond p)$  and  $(t_{32} \diamond p)$ are defined in \eqref{ex:formula}. 
\begin{figure*}[t]
	\begin{equation}
		\label{ex:formula}
	\begin{split}
	& (t_{31} \diamond p)=\frac{((p + 1)(2\xi p - 3\xi^2p - 4\xi p \xi^2 p - 2 \xi p^2 \xi^2p + \xi p^2 + 2))}{(2p - 3\xi p - 4p \xi p - 2 p^2 \xi p + p^2 + 2)} \\
	& (t_{32} \diamond p)=\frac{(10p + 5\xi p - 13\xi^2 p + 6p\xi p - 16p\xi^2p - 10\xi p \xi^2 p + 3p^2\xi p - 8p^2 \xi^2p + 5 p^2 - 6 p^2 \xi p \xi^2p - 12 p \xi p \xi^2p + 8}{(2p - 3\xi p - 4p \xi p - 2 p^2 \xi p + p^2 + 2)} 
\end{split}
	\end{equation}
\end{figure*}
In the expression above, the inverse is understood as the inverse of a matrix over the field of meromorphic functions of the variables $p,qp,\ldots$. The expression $(T \diamond p)$ and $(O \diamond p)$ were computed using Matlab symbolic toolbox.
For any scheduling signal $p \in \mathcal{P}$, define  
\[ (A \diamond p)=A_0+A_1 p, ~ (B \diamond p)=B_0+B_1 p, ~(C \diamond p)=C_0+C_1 p, \]
i.e., $(A \diamond p)(t)=A(p(t))$, $(B \diamond p)(t)=B(p(t))$ and $(C \diamond p)(t)=C(p(t))$
for all $t \in \mathbb{T}$. 
Then  the matrices of the transformed system are of the form
\[ 
  \begin{split}
 &   \stackrel{\leftarrow}{(T \diamond p)} (A \diamond p) (T \diamond p)^{-1} = \begin{bmatrix} (A^{\mathrm o} \diamond p)  & 0  \\
                                                                \star & \star  \end{bmatrix},  \\
	  &  \stackrel{\leftarrow}{(T \diamond p)} (B \diamond p) = \begin{bmatrix} (B^{\mathrm o} \diamond p) \\ \star \end{bmatrix}, ~ \\
		  &   (C \diamond p) (T \diamond p)^{-1} = \begin{bmatrix} (C^{\mathrm o} \diamond p) &  0 \end{bmatrix}
  \end{split}
 \]
 where $\stackrel{\leftarrow}{(T \diamond p)}$ is obtained from $(T \diamond p)$ by replacing each occurence of $p$ by $\xi p$ and
 each occurence of $\xi^i p$ by $\xi^{i+1} p$, for all $i=1,2,\ldots$,
 and  $(A^{\mathrm o} \diamond p)$, $B^{\mathrm o} \diamond p)$ and $(C^{\mathrm o} \diamond p)$ are $2 \times 2$, $2 \times 1$ and $1 \times 2$ matrices with meromophic entries respectively, and $\star$ stands for arbitary matrix.
 The matrices  $(A^{\mathrm o} \diamond p)$, $B^{\mathrm o} \diamond p)$ and $(C^{\mathrm o} \diamond p)$ are $2 \times 2$, $2 \times 1$ and $1 \times 2$ can be found in the .mat files in the supplementary material of this report.

 These matrices depend on $p,\xi p,\xi^2p,\xi^3p$, hence, by a slight abuse of notation,
they can be viewed as matrices of $p,\xi p,\xi^2p,\xi^3p$. 
 According to \cite{Toth2010SpringerBook,Toth11_LPVBehav}  LPV-SS 
 \[
	 \Sigma^{\mathrm o} \left\{
\begin{array}{lcl}
	x (t+1) &=& (A^{\mathrm o} \diamond p)(t) x(t) + (B^{\mathrm o} \diamond p)(t) u(t), \\

	y(t)  &=& (C^{\mathrm o} \diamond p)(t) x(t),
  \end{array}\right.
 \]	 
 has the same manifest behavior as $\Sigma$ and it is a minimal one.
 Note that unlike $\Sigma$, $\Sigma^{\mathrm o}$ has a non-linear and dynamic depency on
 the scheduling variable. To illustrate this, we present $(A^{\mathrm o} \diamond p)_{1,1}$ in \eqref{ex:formula2}.
 \begin{figure*}
\begin{equation}
\label{ex:formula2}	
	(A^{\mathrm o} \diamond p)_{1,1}=\frac{((p + 1)(6\xi^{2}p - 7\xi p - 6p \xi p + 3 p \xi^{2}p + 8\xi p \xi^{2}p - p \xi p^2 - 2p^2 \xi p + 4 \xi p^2 \xi^{2}p + p^2 - 2 \xi p^2 + 2 p \xi p^2 \xi^{2}p + 4 p \xi p \xi^{2}p - 2))}{(2p - 3\xi p - 4p \xi p - 2p^2 \xi p + p^2 + 2)}
\end{equation}
 \end{figure*}
 The other entries of the $A^{\mathrm o} \diamond p$ are even more involved.
 In accordance with Theorem \ref{min:compare}, the LPV-SSA $\Sigma_m$ is also a minimal if viewed as a meromorphic LPV-SS, and 
 it has the same manifest behavior as $\Sigma^{\mathrm o}$. Hence, by \cite{Toth2010SpringerBook,Toth11_LPVBehav} are related
 by a state-space isomorphism which depends on the scheduling variable. We computed the corresponding transformation, it can be found in the .mat file in the supplementary material of this report. 

 The example above demonstrates that if the minimization procedure of  \cite{Toth2010SpringerBook,Toth11_LPVBehav} is applied to
 LPV-SSAs, it will result in a meromoprhic LPV-SS with a dynamic depency. Moreover, even for a simple example, the dependency can get
 quite involved. This is in contrast to the minimization procedure of this paper, which preserves static and affine dependency. 
\end{example}

\section{Proofs of the result}
\label{sect:proof_main}

\subsection{Auxiliary results: observability revealing scheduling}
 The proofs rely on the \textcolor{black}{following observation, which 
 states that for any observable LPV-SSA which satisfies RC,
 there exists a scheduling signal such that the output response to that scheduling signal determines the initial state uniquely. }
%
  \begin{theorem}
\label{min:compare:col1}
 Let $\Sigma$ be an observable LPV-SSA which satisfies RC.
 There exists a scheduling signal $p_\mathrm{o} \in \mathcal{P}$ and $t_\mathrm{o}  \in \mathbb{T}$ such that for any two initial 
	  states $x_1,x_2$ of $\Sigma$, 
	  \[ \textcolor{black}{\mathfrak{Y}_{\Sigma,x_1}(0,p_\mathrm{o})|_{[0,t_{\mathrm o}]}  = \mathfrak{Y}_{\Sigma,x_2}(0,p_\mathrm{o})|_{[0,t_{\mathrm o}]} \implies
	  \mathfrak{Y}_{\Sigma,x_1}=\mathfrak{Y}_{\Sigma,x_2}} \]
 In CT, $p_\mathrm{o}$ can be chosen to be analytic.
\end{theorem}
 \textcolor{black}{The proof relies on viewing LPV-SSAs with zero input as bilinear systems whose inputs are the scheduling signals. Then the existence of $p_{\mathrm o}$ follows from the existence of a universal input for bilinear systems \cite{WangIO,SussmannUnivObs}.
 The RC condition is necessary for using
 \cite{WangIO,SussmannUnivObs}.}
 
\begin{proof}[Proof of Theorem \ref{min:compare:col1}]
 Let $\Sigma$ be of the form \eqref{equ:alpvss}.
 since
 It is enough to show that there exists $t_{\mathrm o} \textcolor{black}{\in \mathbb{T}}$, $p_o \in \mathcal{P}$,
	such that 
  \textcolor{black}{if $\mathfrak{Y}_{\Sigma,x_1}(0,p_o)|_{[0,t_{\mathrm o}]} =\mathfrak{Y}_{\Sigma,x_2}(0,p_o)|_{[0,t_{\mathrm o}]}$.
  then $\mathfrak{Y}_{\Sigma,x_1}(0,p)=\mathfrak{Y}_{\Sigma,x_2}(0,p)$ for all $p \in \mathcal{P}$.}
  \textcolor{black}{Indeed, the latter equality implies} 
  $\mathfrak{Y}_{\Sigma,x_1}=\mathfrak{Y}_{\Sigma,x_2}$, as
   $\mathfrak{Y}_{\Sigma,x_i}(u,p)=\mathfrak{Y}_{\Sigma,x_i}(0,p)+\mathfrak{Y}_{\Sigma,0}(u,p)$ for all
 $i=1,2$, $u \in \mathcal{U}$.
To this end, consider the bilinear system
 \begin{equation}
 \label{eq:bilin:aux}
    \begin{split}
	    & \underbrace{\xi \begin{bmatrix} x(t) \\ z(t) \end{bmatrix}}_{\xi \zeta(t)} = \begin{bmatrix} A(p(t)) & 0 \\ C(p(t)) & 0 \end{bmatrix} \underbrace{\begin{bmatrix} x(t) \\ z(t) \end{bmatrix}}_{\zeta(t)}, 
	  ~ s(t)=z(t) 
   \end{split}
 \end{equation}
	with input $p \in \mathcal{P}$ and output $s \in \mathcal{Y}$.
  For any $p \in \mathcal{P}$,
 denote by $s((x_{\mathrm o},z_{\mathrm o}),p)$ respectively by $\zeta((x_{\mathrm o},z_{\mathrm o}),p)$  the output respectively state
trajectory of \eqref{eq:bilin:aux} 
  generated from the initial state $(x^T_{\mathrm o},z^T_{\mathrm o})^T \in \textcolor{black}{\X \times \mathbb{R}^{\NY}}$ \textcolor{black}{under input $p$}.
  We will call \eqref{eq:bilin:aux} observable, if 
  for each pair of distinct states $(x_1,z_1) \ne (x_2,z_2)$, there exists
	$p \in \mathcal{P}$ such that $s((x_1,z_1),p) \ne s((x_2,z_2),p)$. Note that \eqref{eq:bilin:aux} is observable, if and only if
	$\Sigma$ is observable.
  Indeed, 
  $\delta s((x_{\mathrm o},z_{\mathrm o}),p)=\mathfrak{Y}_{\Sigma,x_{\mathrm o}}(p,0)$ and $s((x_{\mathrm o},z_{\mathrm o}),p)(0)=z_{\mathrm{o}}$.
  Hence, if $\Sigma$ is observable and there exists $(x_1,z_1) \ne (x_2,z_2)$ 
  such that $s((x_1,z_1),p) = s((x_2,z_2),p)$ for every $p \in \mathcal{P}$, then
  $z_1=z_2$ and $\mathfrak{Y}_{\Sigma,x_1}(p,0)=\mathfrak{Y}_{\Sigma,x_2}(0,p)$ for all 
  $p \in \mathcal{P}$. The latter implies that $x_1=x_2$ by observability of $\Sigma$.
  Conversely, if \eqref{eq:bilin:aux} is observable, but there exists 
  $x_1 \ne x_2$ such that $\mathfrak{Y}_{\Sigma,x_1}(p,0) = \mathfrak{Y}_{\Sigma,x_2}(0,p)$ for all
  $p \in \mathcal{P}$, then 
  $s((x_1,0),p)=s((x_2,0),p)$ for all 
  $p \in \mathcal{P}$. The latter contradicts to observability of \eqref{eq:bilin:aux}.


\textcolor{black}{In CT, \TR{let} us  take any $t_{\mathrm o} \in \mathbb{T}$, and let  $p_\mathrm{o}$ be the analytic universal input from  \cite[Theorem 2.11]{SussmannUnivObs} applied to \eqref{eq:bilin:aux}. }
 \textcolor{black}{Note that 
 if $\mathbb{P}$ is a convex set with a non-empty interior, then $\mathbb{P}$ satisfies \cite[Condition H4]{SussmannUnivObs} by \cite[Corollary 2.3.9]{WebsterBook}. }
 For the DT case, let $p_{\mathrm o}$ and $t_{\mathrm o}$ be such that $p_{\mathrm o}|_{[0,t_{\mathrm o}]}$ is the universal
 input $\bar{\nu}$ from %
 \cite[page 1120, proof of Theorem 5.3]{WangIO},
 applied to \eqref{eq:bilin:aux}.  
 The proof of 
 \cite[Theorem 5.3]{WangIO} requires observability of \eqref{eq:bilin:aux} and the following property.
 For any two distinct initial states 
 of \eqref{eq:bilin:aux},
 and any input $p \in \mathcal{P}$, and  time $t \in \mathbb{T}$,
 if the outputs generated from these two initial states are equal on $[0,t]$, then the states of \eqref{eq:bilin:aux} at time $t$ reached from these initial states
 are distinct. More precisely, for  any two $(x_{\mathrm{o},1},z_{\mathrm{o},1}) \ne (x_{\mathrm{o},2},z_{\mathrm{o},2})$
which  satisfy $s((x_{\mathrm{o},1},z_{\mathrm{o},1}),p)|_{[0,t]}=s((x_{\mathrm{o},2},z_{\mathrm{o},2}),p)|_{[0,t]}$, 
it holds that
$\zeta((x_{\mathrm{o},1},z_{\mathrm{o},1}),p)(t) \ne \zeta((x_{\mathrm{o},1},z_{\mathrm{o},1}),p)(t)$.
The latter is assured by invertability of $A(\mathbf{p})$, $\mathbf{p} \in \mathbb{P}$. Indeed, 
from  $s((x_{\mathrm{o},1},z_{\mathrm{o},1}),p)|_{[0,t]}=s((x_{\mathrm{o},2},z_{\mathrm{o},2}),p)|_{[0,t]}$, it follows that
$z_{\mathrm{o},1}=s((x_{\mathrm{o},1},z_{\mathrm{o},1}),p)(0)=s((x_{\mathrm{o},2},z_{\mathrm{o},2}),p)(0)=z_{\mathrm{o},2}$
Hence, $(x_{\mathrm{o},1},z_{\mathrm{o},1}) \ne (x_{\mathrm{o},2},z_{\mathrm{o},2})$ implies $x_{\mathrm{o},1} \ne x_{\mathrm{o},2}$.
Notice that $\zeta((x_{\mathrm{o},i},z_{\mathrm{o},i}),p)(t)=\begin{bmatrix} (\Pi_{s=0}^{t-1} A(p(s)) x_{\mathrm{o},i})^{\top} & (C(p(t-1))\Pi_{s=0}^{t-2} A(p(s)) x_{\mathrm{o},i})^{\top} \end{bmatrix}^{\top}$, $i=1,2$, and as $A(p(s))$ is invertable for all $s \in [0,t-1]$, 
it follows that 
$\zeta((x_{\mathrm{o},1},z_{\mathrm{o},1}),p)(t) \ne \zeta((x_{\mathrm{o},1},z_{\mathrm{o},1}),p)(t)$.
Finally, the 
observability of \eqref{eq:bilin:aux} follows from that of $\Sigma$.

 
 \textcolor{black}{Then  for any two initials states 
 of \eqref{eq:bilin:aux}, if the outputs from those initial states are the same on $[0,t_{\mathrm o}]$ for the input $p_{\mathrm o}$, then the outputs are the same for any input  $p \in \mathcal{P}$ and any time interval.
 Since $\xi s((x_{\mathrm o},z_{\mathrm o})),p)=\mathfrak{Y}_{\Sigma,x_{i,\mathrm o}}(0,p)$ 
 for any $(x_{\mathrm o}^{\top},z^{\top}_{\mathrm o})^{\top} \in \textcolor{black}{\X \times \mathbb{R}^{\NY}}$ and
 $p \in \mathcal{P}$,}
 it follows that $p_{\mathrm o}$ and $t_{\mathrm o}$ satisfy the statement of the theorem. 
\end{proof}

\textcolor{black}{We will also need the following corollary} of Theorem \ref{min:compare:col1}. 
 \begin{corollary}
\label{univ:input}
 Assume that $\Sigma$, $\hat{\Sigma}$ satisfy RC.
	 Then there exists a scheduling signal $p_{\mathrm o} \in \mathcal{P}$ and a time instance \textcolor{black}{$t_{\mathrm o} \in \mathbb{T}$} such that
	 for any two initial states $x_{\mathrm o}$, $\hat{x}_{\mathrm o}$ of $\Sigma$ and $\hat{\Sigma}$ respectively, if
	 $\mathfrak{Y}_{\Sigma,x_{\mathrm o}}(0,p_{\mathrm o})|_{[0,t_{\mathrm o}]}=\mathfrak{Y}_{\hat{\Sigma},\hat{x}_{\mathrm o}}(0,p_{\mathrm o})|_{[0,t_{\mathrm o}]}$, 
then 
	    $\left(\forall p \in \mathcal{P}: \mathfrak{Y}_{\Sigma,x_{\mathrm o}}(0,p)=\mathfrak{Y}_{\hat{\Sigma},\hat{x}_{\mathrm o}}(0,p) \right)$.
\end{corollary}
 \begin{proof}[Proof of \textcolor{black}{Corollary} \ref{univ:input}]
	Let $\Sigma$ be as in  \eqref{equ:alpvss} 
	and $\hat{\Sigma}=(\mathbb{P}, \{\hat{A}_i,\hat{B}_i,\hat{C}_i,\hat{D}_i\}_{i=0}^{\mathrm{n}_p})$.
	Consider the   LPV-SSA 
	 $$\Sigma_c=(\mathbb{P}, \{\begin{bmatrix} A_i & 0 \\ 0 & \hat{A}_i \end{bmatrix}, \begin{bmatrix} B_i \\ \hat{B}_i \end{bmatrix}, \begin{bmatrix} C_i & -\hat{C}_i \end{bmatrix}, D_i-\hat{D}_i \}_{i=0}^{\mathrm{n}_p
  }). $$
    
	Then $\Sigma_c$ satisfies RC, and \textcolor{black}{the i/o functions of
    $\Sigma_c$ are the difference between the i/o functions
    of $\Sigma$ and those of $\hat{\Sigma}$}. That is,
	for $x_{\mathrm o,c}=\begin{bmatrix} x^{\top}_{\mathrm o} & \hat{x}^{\top}_{\mathrm o} \end{bmatrix}^{\top}$,  
	$\mathfrak{Y}_{\Sigma_c,x_{\mathrm o,c}}=\mathfrak{Y}_{\Sigma,x_{\mathrm o}}-\mathfrak{Y}_{\hat{\Sigma},\hat{x}_{\mathrm o}}$.
	Let $\Sigma_{c,\mathrm O}$ be the result of applying 
	Procedure \ref{LSSobs} to $\Sigma_c$. 
	By Remark \ref{LSSobs:rem}, 
	$\mathfrak{Y}_{\Sigma_c,x_{\mathrm o,c}}=\mathfrak{Y}_{\Sigma_{c,\mathrm O},\Pi x_{\mathrm O,c}}$ 
	and $\Sigma_{c,\mathrm O}$ satisfies RC. 
 \begin{color}{black}
        Consider
	the scheduling signal $p_\mathrm{o} \in \mathcal{P}$ and time instance $t_{\mathrm o}$ from 
    Let us apply
	Theorem \ref{min:compare:col1} to $\Sigma_{c,\mathrm O}$
    and let us
	   We claim that  $p_{\mathrm o}$, $t_{\mathrm o}$  satisfy the conclusion of Theorem \ref{univ:input}. 
	Indeed,  assume that $\mathfrak{Y}_{\Sigma,x_{\mathrm o}}(0,p_{\mathrm o})|_{[0,t_{\mathrm o}]}=\mathfrak{Y}_{\hat{\Sigma},\hat{x}_{\mathrm o}}(0,p_{\mathrm o})|_{[0,t_{\mathrm o}]}$ holds. 
    Then the output $\mathfrak{Y}_{\Sigma_{c},x_{\mathrm o,c}}(0,p_{\mathrm o})$ of the  error system $\Sigma_c$ is zero on  $[0,t_{\mathrm o}]$, where
    $x_{\mathrm o,c}=\begin{bmatrix} 
    x^{\top}_{\mathrm o} & \hat{x}^{\top}_{\mathrm o} \end{bmatrix}^{\top}$.
    Hence,  $\mathfrak{Y}_{\Sigma_{c,\mathrm O},\Pi x_{\mathrm o,c}}(0,p_{\mathrm o})|_{[0,t_{\mathrm o}]}=0=\mathfrak{Y}_{\Sigma_{c,\mathrm O},0}(0,p_{\mathrm o})|_{[0,t_{\mathrm o}]}$. 
    From Theorem \ref{min:compare:col1} it  then follows that
    $\mathfrak{Y}_{\Sigma_{c,\mathrm O},\Pi x_{\mathrm o,c}}=\mathfrak{Y}_{\Sigma_{c,\mathrm O},0}$.
	Hence,  
	$\mathfrak{Y}_{\Sigma,x_{\mathrm o}}(0,p)-\mathfrak{Y}_{\Sigma,\hat{x}_{\mathrm o}}(0,p)=\mathfrak{Y}_{\Sigma_{c},x_{\mathrm o,c}}(0,p)=\mathfrak{Y}_{\Sigma_{c,\mathrm O},\Pi x_{\mathrm o,c}}(0,p)=\mathfrak{Y}_{\Sigma_{c,\mathrm O},0}(0,p)=0$ for any
 $p \in \mathcal{P}$.
 \end{color}
\end{proof}
\subsection{Behaviors vs. i/o functions: proof of  Theorem \ref{col:behavior}}
	The implication $\mathbb{F}(\Sigma)=\mathbb{F}(\hat{\Sigma}) \implies  \mathcal{B}(\hat{\Sigma})=\mathcal{B}(\Sigma)$ is obvious.
	 We show the reverse implication. Assume that $\mathcal{B}(\hat{\Sigma})=\mathcal{B}(\Sigma)$. 
	 Consider the scheduling signal $p_{\mathrm o}$ and the time instance $t_{\mathrm o} > 0$ from \textcolor{black}{Corollary} \ref{univ:input}.

	 First, we show that $\mathfrak{Y}_{\Sigma,0}=\mathfrak{Y}_{\hat{\Sigma},0}$.
	 Consider any $p \in \mathcal{P}$, $u \in \mathcal{U}$ and  $\mathbb{T} \ni \tau > 0$ and define
	 $\tilde{u}$ and $\tilde{p}$ such that 
  \textcolor{black}{for all $s \in \mathbb{T}$,
  $\tilde{u}(s+t_{\mathrm o}+\tau)=u(s)$, $\tilde{p}(s+t_{\mathrm o}+\tau)=p(s)$, and }
	 \textcolor{black}{$\tilde{u}|_{[0,t_{\mathrm o}+\tau)}=0$} and $\tilde{p}|_{[0,t_{\mathrm o}]}=p_{\mathrm o}$. 
 \begin{color}{black}
	 Let $\tilde{y}=\mathfrak{Y}_{\Sigma,0}(\tilde{u},\tilde{p})$ be the output of $\Sigma$ generated by the zero initial state for input $\tilde{u}$ and  scheduling signal $\tilde{p}$. Since $(\tilde{y},\tilde{u}, \tilde{p})$ 
    belongs to
  $\mathcal{B}(\hat{\Sigma})=\mathcal{B}(\Sigma)$,  
  there exists an initial state $\hat{x}_{\mathrm o}$ 
  such that
	 $\tilde{y}=\mathfrak{Y}_{\hat{\Sigma},\hat{x}_{\mathrm o}}(\tilde{u},\tilde{p})$.
  Since  
  $\tilde{u}$ equals zero on $[0,t_{\mathrm o}]$, and $\hat{y}$ is the output of $\Sigma$ generated from the zero initial state, 
  $\hat{y}$ must be zero on $[0,t_{\mathrm o}]$.
  Since $\hat{p}$  equals $p_{\mathrm o}$ on $[0,t_{\mathrm o}]$,
%
 $\mathfrak{Y}_{\hat{\Sigma},\hat{x}_{\mathrm o}}(0,p_{\mathrm o})|_{[0,t_{\mathrm o}]}=\hat{y}|_{[0,t_{\mathrm o}]}=0$. 
  \textcolor{black}{From Corollary} \ref{univ:input},
  \mbox{it follows that} $\mathfrak{Y}_{\hat{\Sigma},\hat{x}_{\mathrm o}}(0,\bar{p})=\mathfrak{Y}_{\Sigma,0}(0,\bar{p})=0$ for all
  $\bar{p} \in \mathcal{P}$.
  Hence,  $\mathfrak{Y}_{\hat{\Sigma},\hat{x}_{\mathrm o}}(\bar{u},\bar{p})=\mathfrak{Y}_{\hat{\Sigma},\hat{x}_{\mathrm o}}(0,\bar{p})+\mathfrak{Y}_{\hat{\Sigma},0}(\bar{u},\bar{p})=\mathfrak{Y}_{\hat{\Sigma},0}(\bar{u},\bar{p})$ for all $\bar{u} \in \mathcal{U}$.
  
	 In particular,  $\mathfrak{Y}_{\hat{\Sigma},0}(\tilde{u},\tilde{p})=\mathfrak{Y}_{\Sigma,0}(\tilde{u},\tilde{p})=\hat{y}$.
  Let $x_{f},\hat{x}_f$ be the states of $\Sigma$ respectively $\hat{\Sigma}$
  reached from the zero initial state at time $t_{\mathrm o}+\tau$ under
  input $\tilde{u}$ and scheduling signal $\tilde{p}$. 
  Since $\tilde{u}(s+\tau+t_{\mathrm o})=u(s)$, $\tilde{p}=(s+\tau+t_{\mathrm o})=p(s)$, it follows that $\mathfrak{Y}_{\Sigma,x_f}(u,p)(s)=\hat{y}(s+t_{\mathrm o}+\tau)=\mathfrak{Y}_{\hat{\Sigma},\hat{x}_f}(u,p)(s)$ for all
  $s \in \mathbb{T}$.
  However,  the restriction of $\tilde{u}$ to $[0,t_{\mathrm o}+\tau)$ is zero, and hence $x_f$ and $\hat{x}_f$ are zero.
  Hence, $\mathfrak{Y}_{\hat{\Sigma},0}(u,p)=\mathfrak{Y}_{\Sigma,0}(u,p)$, and as  
	 $u$ and $p$ are arbitrary, 
  $\mathfrak{Y}_{\hat{\Sigma},0}=\mathfrak{Y}_{\Sigma,0}$ follows.
\end{color}

Next we show that $\mathbb{F}(\Sigma) \subseteq \mathbb{F}(\hat{\Sigma})$.
	 To this end, let  $x_{\mathrm o}$ be an initial state of $\Sigma$. 
	  If $y=\mathfrak{Y}_{\Sigma,x_{\mathrm o}}(0,p_{\mathrm o})$,
	 then $(y,0,p_{\mathrm o}) \in \mathcal{B}(\Sigma)=\mathcal{B}(\hat{\Sigma})$, and thus there exists an initial state $\hat{x}_{\mathrm o}$ of $\Sigma$ such that
	 $y=\mathfrak{Y}_{\hat{\Sigma},\hat{x}_{\mathrm o}}(0,p_{\mathrm o})$. From \textcolor{black}{Corollary} \ref{univ:input} it follows that 
	 $\mathfrak{Y}_{\Sigma,x_{\mathrm o}}(0,p)=\mathfrak{Y}_{\hat{\Sigma},\hat{x}_{\mathrm o}}(0,p)$ for all $p \in \mathcal{P}$. 
	  Since  
	 $\mathfrak{Y}_{\Sigma,x_{\mathrm o}}(u,p)=\mathfrak{Y}_{\Sigma,x_{\mathrm o}}(0,p)+\mathfrak{Y}_{\Sigma,0}(u,p)$, 
	 $\mathfrak{Y}_{\hat{\Sigma},\hat{x}_{\mathrm o}}(u,p)=\mathfrak{Y}_{\hat{\Sigma},x_{\mathrm o}}(0,p)+\mathfrak{Y}_{\hat{\Sigma},0}(u,p)$, this implies 
	 $\mathfrak{Y}_{\Sigma,x_{\mathrm o}}=\mathfrak{Y}_{\hat{\Sigma},\hat{x}_{\mathrm o}}$. The inclusion  $\mathbb{F}(\hat{\Sigma}) \subseteq \mathbb{F}(\Sigma)$ can be shown similarly.
\subsection{Minimality results: proof of Theorem \ref{the:behavior_min1}--\ref{the:behavior_min2}}
\begin{proof}[Proof of Theorem \ref{the:behavior_min1}]
 First, we show that if $\Sigma$ is a minimal realization of $\mathcal{B}$, then it is observable. Assume that $\Sigma$
	is not observable. Let us apply Procedure \ref{LSSobs} to $\Sigma$. \textcolor{black}{Then} the resulting LPV-SSA $\Sigma_{\mathrm O}$
	has a smaller state-space dimension than $\Sigma$. \textcolor{black}{Remark \ref{LSSobs:rem}} $\Sigma_{\mathrm O}$ is \textcolor{black}{also realization of $\mathcal{B}$.} This contradicts the  minimality of $\Sigma$.

        \textcolor{black}{Next, we prove that observability implies minimality.}
        Consider two LPV-SSA realizations $\Sigma$ and $\hat{\Sigma}$ of $\mathcal{B}$, such that $\Sigma$ and $\hat{\Sigma}$ 
        both satisfy RC. Then  by Theorem \ref{col:behavior}, $\Phi=\mathbb{F}(\Sigma)=\mathbb{F}(\hat{\Sigma})$.
	\textcolor{black}{Define $\NX=\dim(\Sigma)$ and $\hat{n_{\mathrm x}} = \dim(\hat{\Sigma})$.}
   \textcolor{black}{Let} $\SWS(\Sigma)$ and $\SWS(\hat{\Sigma})$
	associated with $\Sigma$ and $\hat{\Sigma}$ respectively, \textcolor{black}{as defined in} \cite[Appendix, Subsection B]{LPVSSAkalman}. 
 Recall from \cite[Definition 2]{LPVSSAkalman} the notion of
	a switched i/o function $\SWS(\mathfrak{F})$ associated with an i/o function $\mathfrak{F} \in \Phi$.
	and recall that the mapping $\mathfrak{F} \mapsto \SWS(\mathfrak{F})$ is injective. 
	 Define $\SWS(\Phi)=\{ \SWS(\mathfrak{F} \mid \mathfrak{F} \in \Phi\}$ and let 
	$\mu:\SWS(\Phi) \rightarrow \textcolor{black}{\mathbb{R}^{\NX}}$ and $\hat{\mu}:\SWS(\Phi) \rightarrow \textcolor{black}{\mathbb{R}^{\hat{\NX}}}$ be such that
	for any $\mathfrak{F} \in \Phi$, 
	the i/o functions of $\Sigma$ and $\hat{\Sigma}$ \textcolor{black}{induced} by the initial state
	$\mu(\SWS(\mathfrak{F}))$ and  $\hat{\mu}(\SWS(\mathfrak{F}))$ respectively are both equal to $\mathfrak{F}$.
	 Then 
	 $(\SWS(\Sigma),\mu)$, $(\SWS(\hat{\Sigma}),\hat{\mu})$ are both realizations of $\SWS(\Phi)$ in the sense of 
	 \cite{PetCocv11}\footnote{\textcolor{black}{The results of 
  \cite{PetCocv11} can readily be extended to DT, 
      by using the relationship from \cite{Pet12} between linear switched systems in DT and rational representations. }}
      Assume that $\Sigma$ is observable. Then by \cite[Theorem 4]{LPVSSAkalman}, $\SWS(\Sigma)$ is observable. Moreover,
	 $\mu$ is surjective. Indeed, for any initial state $x_{\mathrm o}$, the i/o function generated by $\mu(\SWS(\mathfrak{Y}_{\Sigma,x_{\mathrm o}})$
	 and $x_{\mathrm o}$ are the same, and hence by observability of $\Sigma$, $x_{\mathrm o}=\mu(\SWS(\mathfrak{Y}_{\Sigma,x_{\mathrm o}})$. 
	 Then  $(\SWS(\Sigma),\mu)$ is span-reachable in the terminology of \cite{PetCocv11}, and  by \cite{PetCocv11}, $(\SWS(\Sigma),\mu)$
	 is a minimal realization of $\SWS(\Phi)$. \textcolor{black}{Hence, $\dim \SWS(\Sigma)=\dim \Sigma \le \dim \SWS(\hat{\Sigma})=\dim \hat{\Sigma}$, i.e.
  $\Sigma$ is a minimal realization of $\mathcal{B}$.}

	 Assume that $\Sigma$ and $\hat{\Sigma}$ are minimal realizations of $\mathcal{B}$. By the first part of the theorem,
     they are observable, and hence the linear switched systems $\SWS(\Sigma)$ and $\SWS(\hat{\Sigma})$ are observable.
	 \textcolor{black}{From the} argument of the previous paragraph 
	 it follows 
  \textcolor{black}{that} $(\SWS(\hat{\Sigma}),\hat{\mu})$ and  $(\SWS(\Sigma),\mu)$ are minimal realizations
	 of $\SWS(\Phi)$. Hence, by \cite{PetCocv11,Pet12}, they are isomorphic, and by \cite[Theorem 4]{LPVSSAkalman}, $\Sigma$ and $\hat{\Sigma}$
	 are isomorphic too.
\end{proof}
\textcolor{black}{
In order to prove Theorem  \ref{the:behavior_min2}, we will need to relate controllability and observability of LPV-SSAs with 
with that of the linear-time varying state-space representation  (abbreviated as \emph{LTV-SS}) obtained from the LPV-SSA by fixing a
a particular scheduling signal $p$. The latter LTV-SS
is defined as }
\begin{equation*}
 \Sigma(p) \left \{
  \begin{split}	 
	  &\xi x(t)=A(t)x(t)+B(t)u(t) \\
	  & y(t)=C(t)x(t)+D(t)
   \end{split}\right.
 \end{equation*}
 where $A(t)=A(p(t))$, $B(t)=B(p(t))$, $C(t)=C(p(t))$, $D(t)=D(p(t))$.
 \textcolor{black}{Note that the LTV-SS $\Sigma(p)$ 
 depends on the scheduling signal $p$.
 Recall for \cite{Callier91} the notion of observability and controllability of LTV-SS on a given time-interval. } 
 Then Theorem \ref{min:compare:col1} implies the following.
\begin{corollary}
\label{min:compare:col11}
	With the assumptions and notations of Theorem \ref{min:compare:col1}, 
	the LTV-SS $\Sigma(p_{\mathrm o})$ is 
    \textcolor{black}{observable} on $[0,t_\mathrm{o}]$. 
 \end{corollary}
\begin{corollary}
\label{min:compare:col2}
If $\Sigma$ is span-reachable from $x_\mathrm{o}=0$ and  $\Sigma$ satisfies RC, 
 then there exists 
 $p_\mathrm{r} \in \mathcal{P}$ and 
 \textcolor{black}{$t_\mathrm{r} \in \mathbb{T}$}, such that
	the LTV-SS $\Sigma(p_{\mathrm r})$ is \textcolor{black}{controllable} on $[0,t_\mathrm{r}]$. 
\end{corollary}
\begin{proof}[Proof of Corollary \ref{min:compare:col2}]
Consider the dual LPV-SSA
 $\Sigma^{T}=(\mathbb{P},\{(A^T_q,C^T_q, B_q^T, D_q^T)\}_{q=0}^{n_p})$.
 If $\Sigma$ is span-reachable from zero, 
then $\rank \mathcal{R}_{\NX-1}=n_x$,
	where $\mathcal{R}_{\NX-1}$ is the ($\NX-1$)-step extended reachability matrix of $\Sigma$ from $0$ as defined in \cite{LPVSSAkalman}.
 Let $\mathcal{O}_{\NX-1}$ be the ($\NX-1$)-step extended observability matrix of
	$\Sigma^T$ as defined in \cite{LPVSSAkalman}. 
It is clear that $\mathcal{O}_{\NX-1}=\mathcal{R}_{\NX-1}^T$, and hence
 $\rank \mathcal{O}_{\NX-1}=\NX$ and thus $\Sigma^T$ is observable.
 From Corollary \ref{min:compare:col11} it follows that there exist
	$t_{\mathrm o} \in \mathbb{N}$, and $p_{\mathrm o} \in \mathcal{P}$ 
	 such that the LTV-SS $\Sigma^T(p_{\mathrm o})$ is completely observable on $[0,t_{\mathrm o}]$. 
 This LTV system is given
by matrices $A(t)=A^T(p_{\mathrm o}(t))$, $B(t)=C^T(p_{\mathrm o}(t))$, $C(t)=B^T(p_{\mathrm o}(t))$
	The dual of $\Sigma^T(p_{\mathrm o})$ is then completely controllable on $[0,t_{\mathrm o}]$, and it coincides with $\Sigma(p_{\mathrm o})$.
Hence, we can choose $p_{\mathrm e}=p_{\mathrm o}$ and $t_{\mathrm e}=t_{\mathrm o}$.
\end{proof}

\begin{proof}[Proof of Theorem \ref{the:behavior_min2}]
  If $\Sigma$ is observable and span-reachable from zero, then it is a minimal by 
  Theorem \ref{the:behavior_min2}. Conversely, assume that $\Sigma$ is a minimal realization of $\mathcal{B}$. By Theorem \ref{the:behavior_min2} is 
	$\Sigma$ is observable. Le $x_{\mathrm o}$ be an initial state of $\Sigma$. 
	Let $p_{\mathrm o}$, $t_o$ be the scheduling signal and
	time instance from Theorem \ref{min:compare:col1}. 
	\textcolor{black}{Let $y=\mathfrak{Y}_{\Sigma,x_{\mathrm o}}(0,p_{\mathrm o})$.}
 Then $(y,0,p_{\mathrm o}) \in \mathcal{B}$. 
	Note that $(0,0,p_{\mathrm o}) \in \mathcal{B}$. From
	controllability of $\mathcal{B}$ it  follows that there exists a time instance $\tau > 0$, and signals $\tilde{u},\tilde{p},\tilde{y}$ such that
	$(\tilde{y},\tilde{u},\tilde{p}) \in \mathcal{B}$, $(\tilde{y}|_{[0,t_{\mathrm o}]},\tilde{u}(t)|_{[0,t_{\mathrm o}]},\tilde{p}|_{[0,t_{\mathrm o}]})=(0,0,p_{\mathrm o}|_{[0,t_{\mathrm o}]})$, and \textcolor{black}{$(\tilde{y}(s+t_{\mathrm o}+\tau),\tilde{u}(s+t_{\mathrm o}+\tau), \tilde{p}(s+t_{\mathrm o}+\tau))=(y(s),0,p_{\mathrm o}(s))$} for all 
 $s \in \mathbb{T}$. 
	Since $\Sigma$ is a realization of $\mathcal{B}$, there exists an initial state
	$\hat{x}_{\mathrm o}$ of $\Sigma$ such that $\mathfrak{Y}_{\Sigma,\hat{x}_{\mathrm o}}(\tilde{u},\tilde{p})=\tilde{y}$.
	It then follows that $\mathfrak{Y}_{\Sigma,\hat{x}_{\mathrm o}}(0,p_{\mathrm o})|_{[0,t_{\mathrm o}]}=0=\mathfrak{Y}_{\Sigma,0}(0,p_{\mathrm o})|_{[0,t_{\mathrm o}]}$, 
	and hence $\mathfrak{Y}_{\Sigma,0}=\mathfrak{Y}_{\Sigma,\hat{x}_{\mathrm o}}$.
	Since $\Sigma$ is observable, it then follows that  
	$\hat{x}_{\mathrm o}=0$. Consider now the state $\hat{x}_f$ 
   reached from the zero initial state
	using $\tilde{u}$,$\tilde{p}$ at time $t_{\mathrm o}+\tau$.
	It then follows that 
	\textcolor{black}{$\mathfrak{Y}_{\Sigma,\hat{x}_f}(0,p_{\mathrm o})(s)=\mathfrak{Y}_{\Sigma,0}(\tilde{u},\tilde{p})(s+t_{\mathrm o}+\tau)= \mathfrak{Y}_{\Sigma,x_{\mathrm o}}(0,p_{\mathrm o})(s)$, $s \in \mathbb{T}$}. Hence, by Theorem \ref{min:compare:col1}
	and observability of $\Sigma$, $x_{\mathrm o}=\hat{x}_f$, i.e., $x_{\mathrm o}$ is reachable from zero. 
	 That is, $\Sigma$ is span-reachable.

	 \textcolor{black}{Assume that $\Sigma$ a realization of $\mathcal{B}$}
  which is span-reachable from zero.
  Consider 
   \textcolor{black}{$p_c \in \mathcal{P}$ and $t_c \in \mathbb{T}$} from Corollary \ref{min:compare:col2}.
	 Consider two elements $(u_1,p_1,y_1)$ and $(u_2,p_2,y_2)$ of \textcolor{black}{$\mathcal{B}$} and a time instance $t$.
	 Then there exist initial states $x_{\mathrm o,1}, x_{\mathrm o,2}$ such that 
	 $y_i=\mathfrak{Y}_{\Sigma,x_{\mathrm o,i}}(u_i,p_i)$. \textcolor{black}{Let $x_{\mathrm o}$ be the state reached from $x_{\mathrm o,1}$ at time $t^{+}$ under input $u_1$
     and scheduling signal $p_1$,}
	 where $t^{+}=t$ in CT and $t^{+}=t+1$ in DT.
	 \textcolor{black}{Let} $u_c$ and $t_c > \tau \ge 0$ be such that $x_{\mathrm o,2}$ 
  \textcolor{black}{is the state reached from $x_{\mathrm o}$
  at time $\tau^{-}$ under input $u_c$ and scheduling signal $p_c$},
	 where $\tau^{-}=\tau$ in CT and $\tau^{-}=\tau-1$ in DT. 
	 Since $\Sigma(p_c)$ is \textcolor{black}{controllable} on $[0,t_{c}]$, such $u_c$ and $\tau$ exist. 
	 Define $\tilde{u}, \tilde{p}$ such that $\tilde{u}|_{[0,t]}=u_1|_{[0,t]}, \tilde{p}|_{[0,t]}=p_1|_{[0,t]}$,
	 $\tilde{u}(s)=u_c(s-t), \tilde{p}(s)=p_c(s-t)$, $s \in (t,t+\tau)$, and 
	 \textcolor{black}{$\tilde{u}(s+t+\tau)=u_2(s), \tilde{p}(s+t+\tau)=p_2(s)$, $s \in \mathbb{T}$}. 
	 Let $\tilde{y}=\mathfrak{Y}_{\Sigma,x_{\mathrm o,1}}(\tilde{u},\tilde{p})$.
	 Then $(\tilde{u},\tilde{p},\tilde{y}) \in \mathcal{B}$, 
	  $\tilde{u}(s),\tilde{p}(s),\tilde{y}(s))=(u_1(s),p_1(s),y_1(s))$ for $s \le t$, and 
	  $\tilde{u}(s),\tilde{p}(s),\tilde{y}(s))=(u_2(s-t-\tau),p_2(s-t-\tau),y_2(s-t-\tau))$ for $s \le t+\tau$. 
	 That is, $\mathcal{B}$ is controllable.
\end{proof}

\subsection{Relationship with prior work: proof of Theorem \ref{min:compare}}

\begin{proof}[Proof of Theorem \ref{min:compare}]
 We prove the statement  for observability, the 
 statement on span-reachability follows by duality. 
 From Corollary \ref{min:compare:col1} it follows that 
	the LTV-SS  $\Sigma(p_{\mathrm o})$ 
 is observable on $[0,t_{\mathrm o}]$.
 From
 \cite{Callier91} it follows that 
 there exists $k \ge 0$, such that the $k$ step observability
	matrix of $\Sigma(p_{\mathrm o})$ is full column  in DT, and 
	 it full column rank for almost all $t$ on $(0,t_{\mathrm o})$ in the CT case.
From \cite[Definition 3.34]{Toth2010SpringerBook} it then follows that $\Sigma$ is structurally observable.
	The last statement of the theorem follows from \cite[Theorem 3.14]{Toth2010SpringerBook}.
\end{proof}


\section{Conclusions}\label{para:concl} 
  In this paper a characterization of minimal LPV-SSA realizations of LPV behaviors in terms of observability has been presented. It
  has also been  shown that minimal LPV-SSA realizations of the same behavior are isomorphic.
  These results represent the first steps towards a behavioral approach for
  LPV-SSAs. Future work will be directed towards developing counterparts of i/o partitioning, and kernel representations for 
  manifest behaviors of
  LPV-SSAs. 

\bibliographystyle{plain}

\end{document}